\documentclass[12pt,reqno]{amsart}
\usepackage{graphicx}
\usepackage{amsmath}
\usepackage{amssymb}
\usepackage{amscd}
\usepackage{mathrsfs}
\usepackage[all,cmtip]{xy}
 \usepackage{color}

\newtheorem{theorem}{Theorem}

\newtheorem{theoremc}{Theorem}

\newtheorem{rk}[theoremc]{Remark\!\!}

\newtheorem{lem}[theorem]{Lemma}

\renewcommand\1{{\bf 1}}

\newcommand\com[1]{}

\newcommand\Cc{{\let\mathcal\mathscr\mathcal C}}

\newcommand\g{{\frak g}}

\renewcommand\l{\lambda}

\newcommand\La{\Lambda}

\newcommand\op[1]{\mathop{\rm #1}\nolimits}
\newcommand\ot{\otimes}

\newcommand\R{{\mathbb R}}

\newcommand\vp{\varphi}

\newcommand\z{\sigma}

\begin{document}

 \title{Conformal Differential Invariants}
 \author{Boris Kruglikov}
 \date{}
\address{Institute of Mathematics and Statistics, NT-faculty, University of Troms\o, Troms\o\ 90-37, Norway.
\quad E-mail: {\tt boris.kruglikov@uit.no}. }

 \vspace{-14.5pt}
 \begin{abstract}
We compute the Hilbert polynomial and the Poincar\'e function counting the number of fixed jet-order
differential invariants of conformal metric structures modulo local diffeomorphisms, and
we describe the field of rational differential invariants separating generic orbits of the
diffeomorphism pseudogroup action. This resolves the local recognition problem for conformal structures.
 \end{abstract}

 \maketitle

 \section*{Introduction}

Differential invariants play a central role in the classification problems of geometric
structures. Often the fundamental invariants have tensorial character, but for resolution of
the equivalence problem scalar invariants are required to be derived from those.

For instance, the fundamental invariant of a Riemannian metric $g$ on a manifold $M$ is the Riemann
 tensor $R_g\in\Gamma(\La^2T^*M\ot\mathfrak{so}(TM))$.
Scalar differential invariants are Weyl curvature invariants \cite{W}, separating generic orbits of the
diffeomorphism pseudogroup $G=\op{Diff}_\text{loc}(M)$ acting on the space of jets of metrics
$J^\infty(S^2_\text{ndg}T^*M)$, where $S^2_\text{ndg}T^*M$ is the complement in $S^2T^*M$ to the cone
of degenerate quadrics, and they are obtained by contractions
of the tensor products of the covariant derivatives of the curvature tensor $R_g$.
Their number depending on the jet-order was computed by Zorawski \cite{Z} and Haskins \cite{Ha},
see also \cite{T}.

In this paper we do the same for conformal metric structures $(M,[g])$ of arbitrary signature
in dimensions $n=\dim M>2$. Notice that for $n=2$ the conformal group is too large and, due to
Gau\ss\ theorem on existence of isothermal coordinates, there are no local invariants of conformal
structures, and hence no differential invariants in 2D.

The fundamental invariants $C$ of the conformal structure are the Cotton tensor for $n=3$ and
the Weyl tensor for $n>3$. Similarly to Weyl scalar invariants for Riemannian metrics, one could expect
scalar invariants to be derived from the fundamental tensor invariants, and this was done in
\cite{FG,BEG,G}, and will be discussed in the next section. These scalar invariants are however
defined on the (proper jet-lift of the) ambient space $\hat M$ to our $M$, $\dim\hat M=n+2$, so
that the constructed scalars are covariants rather than invariants.

There is however an easy approach to construct differential invariants for generic conformal structures.
It is based on the folklore result that in the domain $U\subset M$, where
$\|C\|^2_g\neq0$ for some (and hence any) representative $g\in[g]$, one can uniquely fix
(actually up to $\pm$ if the signature is split) a metric $g_0$ in the conformal class $[g]$
by the normalization $\|C\|^2_{g_0}=\pm1$ (the sign is always $+$ in the Riemannian case, but can be any
in the indefinite case). Then the conformal invariants are derived from the (pseudo-)Riemannian metric ones
(Weyl curvature invariants or those from \cite{K,LY}).

This however does not yield the number\footnote{We have to fix the signature $(p,q)$ of $[g]$,
$p+q=n$. The formulas for invariants vary a bit with this $(p,q)$, but the number of invariants depends
only on $n$.}  of scalar differential invariants $H_n(k)$
depending on the jet-order $k$ (we count so-called "pure order", see below).
The classical approach to computing these numbers is the Lie method of elimination of group parameters
(or algebra parameters), see \cite{Ha,Z,O}. This involves calculation of ranks of large matrices.
Instead we rely on some simple algebraic ideas and compute the Hilbert polynomial $H_n(k)$, the
first values of which are given below:
\medskip
 \begin{center}
\begin{tabular}{c||c|c|c|c|c|c}
$n\ \backslash\ k$ & 1 & 2 & 3 & 4 & \dots & $k$ \\
\hline\hline
3 \vphantom{$\dfrac{o}o$} & 0 & 0 & 1 & 9 & \dots & $k^2-4$ \\
\hline
4 \vphantom{$\dfrac{o}o$} & 0 & 3 & 36 & 91 & \dots & $\tfrac16(k+2)(k+3)(5k-7)$ \\
\hline
5 \vphantom{$\dfrac{o}o$} & 0 & 24 & 135 & 350 & \dots & $\tfrac1{24}(k+2)(k+3)(k+4)(9k-11)$
\end{tabular}
 \end{center}
\bigskip

Then we derive the Poincar\'e function encoding these numbers.
We also indicate a different set of conformal differential invariants, now rational,
and describe the field they generate.

 \section{The algebras and fields of differential invariants}

The scalar conformal invariants mentioned in the introduction are constructed via the ambient metric
construction of Fefferman and Graham \cite{FG} roughly as follows.
Consider the bundle $\bar M=M\times\R_+$ over $M$ consisting of all representatives $g$ of $[g]$
with its natural horizontal metric $\bar g$ (tautological structure: $\bar g_g=g\circ d_g\pi$, where
$\pi:\bar M\to M$), and let $\hat M=\bar M\times(-1,1)$. The ambient
metric $\hat g$ is $\R_+$-scaling weight 2 homogeneous Ricci flat Lorentzian metric on $\hat M$
restricting to $\bar g$ on $\bar M\times\{0\}$. This exists on the infinite jet of
$\bar M\times\{0\}\subset\hat M$ for odd $n$, and up to order $n/2$ for even $n$.
Taking the Weyl metric curvature invariants of $\hat g$ yields scalar invariants of $[g]$,
which give a complete set of polynomial invariants\footnote{When we write "polynomial" here and beyond
we mean only with respect to jets of order $>0$, allowing division by the determinant of $g$ everywhere.}
for odd $n$ and the same to a finite order for even $n$, see \cite{BEG,G}.
The definite advantage of these invariants is that they are defined for all conformal structures.

There are however two basic problems with these ambient Weyl conformal invariants, similar to the classical
Weyl metric curvature invariants. First of all, the algebra generated by these polynomial invariants
is not finitely generated. Secondly, it is not apriori clear which of these differential invariants are
separating for the orbits of the diffeomorphism pseudogroup action (on infinite or any finite jet-level).

The second problem is solved by passing to rational differential invariants: since the action is algebraic,
its prolongations are algebraic too \cite{KL2}, and in any finite jet-order there exists a rational quotient by
the action due to the Rosenlicht theorem \cite{R}. From this viewpoint the field $\mathfrak{F}$ of rational
differential invariants is useful and simpler. The invariants obtained in this way will be presented below.

The first problem is a bit more complicated, as it is clear that the transcendence degree
$\op{trdeg}(\mathfrak{F})=\infty$, so just passing to rational invariants does not resolve infinite generation.
In the early days of differential invariants theory it was suggested and motivated by Sophus Lie and Arthur Tresse
that the algebra of differential invariants is generated by a finite number of differential invariants
$I_1,\dots,I_t$ and a finite number of invariant derivations $\nabla_1,\dots,\nabla_s$. This was later
proved in several versions, see \cite{KL2} and the references therein.

In more details, consider the algebra $\mathfrak{A}_l$ of differential invariants that are rational by the jets
of order $\leq l$ and polynomial by the jets of higher order ($l$ is determined by the structure in question,
we will see that in the case of conformal structures $l=4$ for $n=3$ and $l=3$ for $n>3$).
This $\mathfrak{A}_l$ is called the algebra of rational-polynomial invariants.

The main result of \cite{KL2} states that $\mathfrak{A}_l$ is finitely generated by $I_i,\nabla_j$, i.e.\
any differential invariant from $\mathfrak{A}_l$ is a polynomial of $\nabla_JI_i$ for ordered multi-indices
$J=(j_1,\dots,j_r)$ with rational coefficients of $I_k$.

Now the field of rational differential invariants $\mathfrak{F}$ is generated by $\mathfrak{A}_l$ for some $l$,
and so is also finitely generated in the Lie-Tresse sense as above. The algebra $\mathfrak{A}_l$ separates
the orbits of the $G$-action on the space of jets of conformal structures $J^\infty(\mathcal{C}_M\!)$, where 
 $$\mathcal{C}_M=S^2_\text{ndg}T^*M/\R_+,$$
and we get a finite separating set of invariants $\nabla_JI_i$, $|J|\le k-\op{deg}(I_i)$ for the restriction
of the action on $J^k$. Thus we obtain a set of generators for the field $\mathfrak{F}_k$ of rational
invariants of order $k$ that filter the field $\mathfrak{F}$.

 \section{Scalar invariants of conformal metric structures}\label{S2}

Let us generate conformal differential invariants of generic conformal structures $[g]$. This will in turn
generate rational differential invariants on the space of jets of all conformal structures
$J^\infty(\mathcal{C}_M\!)$.

We begin with the case $n\ge4$. Since there are no metric invariants of order $<2$, there are no conformal
invariants of lower order too. The lowest order conformal invariants live in 2-jets.
It is well-known that the complete invariant there is the conformal Weyl tensor $C$, considered as a $(3,1)$ tensor.
Indeed, the only invariant of the 2-jet of a Riemannian metric $g$ is the Riemann curvature tensor $R_g$
(due to existence of normal geodesic coordinates),
and conformal re-scalings of $g$ leave invariant only the Weyl part of it \cite{B}.

The space of conformal Weyl tensors $\mathcal{W}$ (at one point) has dimension $\dim\mathcal{W}=\tfrac1{12}(n-3)n(n+1)(n+2)$, the conformal linear group $CO(g)$ acts effectively
on $\mathcal{W}$, and so the codimension of a generic orbits equals
 $$
H_n(2)=\dim\mathcal{W}-\dim CO(n)=\tfrac1{12}(n^4-13n^2-12).
 $$
Since $CO(g)=SO(g)\times\R_+$ is a reductive Lie group acting in algebraic manner on $\mathcal{W}$, the Hilbert
invariant theorem \cite{Hi} implies existence of invariants separating generic orbits.

Moreover, $SO(g)$-invariants can be taken to be polynomial, but presence of $\R_+$-factor
compels to extend to rational invariants (in this case, however dividing by the determinant of metric is enough). 
The geometric invariant theory \cite{M} provides a method to construct these, 
and we can choose among them $H_n(2)$ functionally
independent invariants separating orbits on a Zariski open set in $\mathcal{W}$.

Let us indicate first how to construct non-algebraic invariants, that are obtained by a finite
algebraic extension. Consider a generic element $C\in\mathcal{W}$, at this point meaning only
$\|C\|^2_g\neq0$. Fix a metric $g_0\in[g]$ by the condition $\|C\|^2_{g_0}=\pm 1$. Then we can convert $C$ to
a $(2,2)$ tensor, interpreted as a linear map $C:\La^2T\to\La^2T$, where $T=T_aM$ is the tangent space to
$M$ at the considered point $a$. This map is $g$-symmetric, traceless and has unit norm.
Therefore its spectrum gives $d=\binom{n}2-2$ real scalar
invariants $\lambda_1,\dots\lambda_d$ (alternatively, pass to algebraic invariants $\op{Tr}(C^i)$,
$1<i\le d+1$). Notice that the spectrum $\op{Sp}(C)$ is simple for $C$ from a Zariski open set in $\mathcal{W}$.

Let $\sigma_i\in\La^2T$ be the eigenvectors corresponding to $\lambda_i$, normalized by the
condition\footnote{Here for $\sigma_i$ as well as before for $g_0$ a possible freedom of the sign choice is hidden.
This can be locked, but we prefer to ignore it for simplicity of the exposition.} $\|\z_i\|^2_{g_0}=\pm1$
(the sign again can be arbitrary in the case of indefinite signature of $[g]$; our genericity assumption
implicity implies that $\|\z_i\|^2_{g_0}\neq0$, so $\z_i$ can be rescaled to get the desired normalization).
Then the operators $A_i=g_0^{-1}\z_i$ carry a lot of invariants, for instance $\op{Tr}(A^\z)$ for
$A^\z=A_1^{k_1}\cdots A_d^{k_d}$ when $\z=(k_1,\dots,k_d)$,
and we can extract $H_n(2)$ independent among them.

Let us also notice that the normalized eigenbases of one these ope\-ra\-tors $A^\z$ give us a canonical
frame $e_1,\dots,e_n$ (provided that one of them has simple spectrum - this is yet another requirement for $C$
being generic). This frame depends on the 2-jet of a conformal structure.

 \begin{rk}
The skew-symmetric operators $A_i$ have purely imaginary spectrum, so one has to consider their
products, or work in complexification and take the real parts, to derive non-trivial real invariants.
The case $n=4$ is however an exception ($n=3$ is an exception too, to be considered later).

In this case, due to exceptional isomorphism $\mathfrak{so}(4)=\mathfrak{so}(3)\oplus\mathfrak{so}(3)$,
the algebra of operators is the sum of these two algebras, interpreted as the action of unit purely imaginary
quaternions $\mathbb{S}^2\subset\op{Im}\mathbb{H}$ from left and right on $\mathbb{H}$. Denote these operators
by $J^\text{left}_i$ and $J^\text{right}_i$, $i=1,..,3$. They are in quaternionic relations and the left ones
commute with the right ones. The operators $B_i=J^\text{left}_i\cdot J^\text{right}_i$ have
$\op{Sp}(B_i)=\{\pm1,\pm1\}$, and the corresponding eigenspace decomposition is $T=\Pi^-_i\oplus\Pi^+_i$,
$\dim\Pi^\pm_i=2$. The intersections $\Pi^\pm_i\cap\Pi^\pm_j$ yield the splitting of $T$ into direct sum of 4 lines,
whence the frame $e_1,\dots,e_4$ obtained by $g_0$-normalization (a residual finite symmetry related to numeration
of $e_i$ and change of sign remains here, but can be eliminated on further steps).
 \end{rk}

Writing the canonical representative $g_0\in[g]$ in this frame we obtain all other differential
invariants. This can be formulated in the framework of Lie-Tresse theorem, since $\nabla_j=D_{e_j}$
(horizontal lift to $J^\infty$) form the basis of invariant derivations,
and we can choose $I_i$ among the second order invariants already constructed.

Now we achieve algebraicity as follows. The algebraic extension is given by variables $y_1,\dots,y_p$ that
are in algebraic relations with vertical coordinates on the space $J^2(\mathcal{C}_M\!)$.
These enter both $I_i$ and $\nabla_j$, but the derived differential invariants $\nabla_JI_i$
are algebraic by higher order jets. Considering the algebra of invariants generated by these on both second
and higher jets, we can eliminate (for instance, via a Gr\"obner basis) the $y$-variables and get
a system of separating algebraic invariants that generate the fields $\mathfrak{F}_k$ for $k=2$ and $k>2$
by the Rosenlicht theorem. We can also eliminate $y$-parameters in the coefficients of
$\nabla_i=\kappa_i^jD_j$ (by taking linear combinations with invariant coefficients) to have
rational invariant derivations.

It is not difficult to see (it also follows from considerations in the next section) that the denominators
in the rational differential invariants can be chosen supported in the 3rd jets for $n>3$,
and in the 4th jets for $n=3$. Hence we can take $l=3$ in the algebra $\mathfrak{A}_l$ of
rational-polynomial invariants discussed in the previous section for $n>3$, and $l=4$ for $n=3$.
This establishes the following statement for $n\ge4$.

 \begin{theorem}
The algebra $\mathfrak{A}_l$ of rational-polynomial invariants for $l={3+\delta_{n3}}$ as well the field
$\mathfrak{F}$ of rational differential invariants of conformal metric structures
are both generated by a finite number of (the indicated) differential invariants $I_i$
and invariant derivations $\nabla_j$, and the invariants from this
algebra/field separate generic orbits in $J^\infty(\mathcal{C}_M\!)$.
 \end{theorem}

Let us consider now the exceptional case $n=3$ and justify the above theorem in this case.
There are no conformal invariants in 3D of order 2, and all
differential invariants of order 3 are derived from the Cotton tensor $C$ considered as $(3,0)$ tensor.
The space of Cotton tensors has dimension 5, is acted upon effectively by $CO(g)$ of dimension 4,
so the generic orbit has codimension $H_3(3)=1$
(this fact was also checked independently by a straightforward computation in Maple).

For generic $C$ we have $\|C\|^2_g\neq0$, and so we can fix the metric representative $g_0\in[g]$ by
$\|C\|^2_{g_0}=\pm1$. Then we convert $C$, using the Hodge $*$-operator of $g_0$, to the $(1,1)$ Cotton-York
tensor $C:T\to T$. Again by genericity the spectrum $\op{Sp}(C)=\{\l_1,\l_2,\l_3\}$ is simple,
and the relations $\sum\lambda_i=0$, $\max|\lambda_i|=1$ yield precisely one scalar invariant of order 3;
we can take, for instance, the polynomial invariant $\op{Tr}(C^2)$.

In addition, we have the ($g$-normalized) eigenbasis $e_1,e_2,e_3$
(that depends on the 3-jet of generic conformal structures).
This produces invariant derivations $\nabla_1,\nabla_2,\nabla_3$, as before,
and writing $g_0$ in this frame we get all 4th and higher order differential
invariants $I_i$ sufficient for Lie-Tresse generating property. These invariants will be indeed separating,
and eliminating non-algebraicity as before, we derive the fields of rational invariants $\mathfrak{F}_k$
with $\cup_k\mathfrak{F}_k=\mathfrak{F}$.


 \section{Stabilizers of generic jets}\label{S3}

Our method to compute the number of independent differential invariants of order $k$ follows the approach
of \cite{LY}. We will use the jet-language from the formal theory of PDE, and refer the reader to \cite{KL1}.

Fix a point $a\in M$. Denote by $\mathbb{D}_k$ the Lie group of $k$-jets of diffeomorphisms
preserving the point $a$. This group is obtained from $\mathbb{D}_1=\op{GL}(T)$ by successive
extensions according to the exact 3-sequence
 $$
0\to\Delta_k \longrightarrow \mathbb{D}_k\longrightarrow \mathbb{D}_{k-1}\to\{e\},
 $$
where $\Delta_k=\{[\vp]_x^k:[\vp]_x^{k-1}=[\op{id}]_x^{k-1}\}\simeq S^kT^*\otimes T$ is Abelian ($k>1$).

\smallskip

Denote $V_M=T^\text{\rm vert}_{[g]}(\mathcal{C}_M\!)$ the tangent to the fiber of $\mathcal{C}_M$.

 \begin{lem}
The following is a natural isomorphism:
$$V_M=\op{End}^\text{\rm sym}_0(T)=\{A:T\to T\,|\, g(Au,v)=g(u,Av),\op{Tr}(A)=0\}.$$
 \end{lem}

 \begin{proof}
For a curve $[g+\epsilon\,\sigma]=[g(\1+\epsilon\,g^{-1}\z)]$ in $\mathcal{C}_M$ let us associate to its tangent
vector the endomorphism $A=g^{-1}\z-\frac1n\op{Tr}(g^{-1}\z)\1$. Since
 $$
\1+\epsilon g^{-1}\z=\bigl(1+\tfrac{\epsilon}n\op{Tr}(g^{-1}\z)\bigr)\1+\epsilon A
=\bigl(1+\tfrac{\epsilon}n\op{Tr}(g^{-1}\z)\bigr)\cdot\Bigl(\1+
\tfrac{\epsilon n}{n+\epsilon\!\op{Tr}(g^{-1}\z)}A\Bigr)
 $$
removal of the trace part of $\sigma$ (equivalent to conformal rescaling of the representative)
is in the kernel of this map. Since $A$ is obviously $g$-symmetric (for any representative $g$),
this map is the required isomorphism.
 \end{proof}

Thus the symbol of the bundle $J^k(\mathcal{C}_M\!)$ is
 $$
\g_k=\op{Ker}[d\pi_{k,k-1}:TJ^k(\mathcal{C}_M\!)\to TJ^{k-1}(\mathcal{C}_M\!)]=S^kT^*\otimes V_M.
 $$
The differential group $\mathbb{D}_{k+1}$ acts on $J^k_a(\mathcal{C}_M\!)$,
and hence $\Delta_{k+1}$ acts on $\g_k$. Let $a_k\in J^k_a(\mathcal{C}_M\!)$ be a generic point.
The next statement is obtained by a direct computation of the symbol of Lie derivative.

 \begin{lem}
The space $\Delta_{k+1}\cdot a_k\subset\g_k$ is the image $\op{Im}(\zeta_k)$ of the map
$\zeta_k$ that is equal to the following composition
 $$
S^{k+1}T^*\ot T\stackrel{\delta}\longrightarrow S^kT^*\otimes(T^*\ot T)
\stackrel{\1\otimes\Pi}\longrightarrow S^kT^*\ot V_M.
 $$
Here $\delta$ is the Spencer operator and $\Pi:T^*\ot T\to V_M\subset T^*\ot T$
is the projection given by
 $$
\langle p,\Pi(B)u\rangle=\tfrac12\langle p,Bu\rangle
+\tfrac12\langle u_\flat,Bp^\sharp\rangle-\tfrac1n\op{Tr}(B)\langle p,u\rangle,
 $$
where $u\in T,p\in T^*,B\in T^*\ot T$ are arbitrary,
$\langle\cdot,\cdot\rangle$ denotes the pairing between $T^*$ and $T$,
and we use the musical isomorphisms $\flat$ (flat) and $\sharp$ (sharp) that depend on the choice of
representative $g\in[g]$, but the right-hand side is independent of it.
In the index notations:
 $$
\Pi(B)^i_j=\tfrac12(B^i_j+g^{ik}B^l_kg_{lj})-\tfrac1n B^k_k\delta^i_j.
 $$
 \end{lem}

One should, of course, check that the image of $\Pi$ belongs to $V_M$, but this is straightforward.
Recall that $i$-th prolongation of a Lie algebra $\mathfrak{h}\subset\op{End}(T)$
is defined by the formula $\mathfrak{h}^{(i)}=S^{i+1}T^*\ot T\cap S^iT^*\ot\mathfrak{h}$.
As is well-known, for the conformal algebra of $[g]$ of any signature $(p,q)$, $n=p+q>2$, it holds:
$\mathfrak{co}(g)^{(1)}=T^*$ and $\mathfrak{co}(g)^{(i)}=0$ for $i>1$.

 \begin{lem}\label{Lzeta}
We have: $\op{Ker}(\zeta_k)=0$ for $k>1$.
 \end{lem}

 \begin{proof}
If $\zeta_k(\Psi)=0$, then $\delta(\Psi)\in S^kT^*\ot\mathfrak{co}(g)$, where
$\mathfrak{co}(g)\subset\op{End}(T)$ is the conformal algebra. This means that
$\Psi\in\mathfrak{co}(g)^{(k+1)}=0$, if $k>1$. Thus we conclude injectivity of $\zeta_k$.
 \end{proof}

Denote by $\op{St}_k\subset\mathbb{D}_{k+1}$ the stabilizer of a generic point
$a_k\in J^k_a(\mathcal{C}_M\!)$, and by $\op{St}^0_k$ its connected component of unity.
Then Lemma \ref{Lzeta} implies that $\Delta_{k+1}\cap\op{St}_k=\{e\}$ for $k>1$, so
the projectors $\rho_{k+1,k}:\mathbb{D}_{k+1}\to\mathbb{D}_k$ induce the injective homomorphisms
$\op{St}_k\to\op{St}_{k-1}$ and $\op{St}^0_k\to\op{St}^0_{k-1}$.

The stabilizers of low order (for any $n\ge3$) are the following. For any $a_0\in \mathcal{C}_M$
its stabilizer is $\op{St}_0=CO(g)$. Next, the stabilizer $\op{St}_1\subset\mathbb{D}_2$ of
$a_1\in J^1(\mathcal{C}_M\!)$ is the extension
(by derivations) of $\op{St}_0$ by $\mathfrak{co}(g)^{(1)}=T^*\stackrel{\iota}\hookrightarrow\Delta_2$,
where $\iota:T^*\to S^2T^*\ot T$ is given by
 $$
\iota(p)(u,v)=\langle p,u\rangle v+\langle p,v\rangle u
-\langle u_\flat,v\rangle p^\sharp,
 $$
for $p\in T^*$, $u,v\in T$,
or by using indices: $\iota(p)^j_{kl}=p_k\delta^j_l+p_l\delta^j_k-g^{ij}p_ig_{kl}$.
In other words, we have $\op{St}_1=CO(g)\ltimes T$.

Notice also that for $n=3$ due to absence of second order differential invariants
and equality $\dim\Delta_3=\dim\g_2$ we have $\op{St}^0_2=\op{St}^0_1$.
Then by dimensional reasons $\op{St}_2$ for $n\ge4$ and $\op{St}_3$ for $n=3$ are nontrivial
($\dim\mathbb{D}_3>\dim J^2_a(\mathcal{C}_M\!)-H_n(2)$ for $n>3$, resp.
$\dim\mathbb{D}_4>\dim J^3_a(\mathcal{C}_M\!)-H_3(3)$ for $n=3$).

 \begin{lem}\label{St}
If $k\ge3$, $n\ge4$ of if $k\ge4$, $n=3$, then $\op{St}^0_k=\{e\}$.
 \end{lem}

 \begin{proof}
In Section \ref{S2} we constructed a canonical frame $e_1,\dots,e_n$ on $T$ depending on
(generic) jet $a_i\in J^i(\mathcal{C}_M\!)$, where $i=2$ for $n>3$ and $i=3$ for $n=3$. In other words,
we constructed a frame on the bundle $\pi_i^*TM$ over a Zariski open set in $J^i(\mathcal{C}_M\!)$.

The elements from $\op{St}^0_i$ shall preserve this frame.
Since $\op{St}^0_i\subset\op{St}_1=CO(g)\ltimes T$ this eliminates the liner conformal freedom (first factor).

Choosing a point $a_{i+1}\in\pi^{-1}_{i+1,i}(a_i)\in J^{i+1}(\mathcal{C}_M\!)$
the elements of $\op{St}^0_{i+1}\subset\op{St}_1$ should also stabilize 1-jet of this frame
(realized via the canonical lift $L(a_{i+1})\subset T_{a_i}J^i(\mathcal{C}_M\!)$ of $T=T_aM$),
and this eliminates the remaining freedom $\mathfrak{co}(g)^{(1)}=T$, yielding
$\op{St}^0_{i+1}=0$ (we take the connected component because of the undetermined signs $\pm$
in the normalizations above). Hence the stabilizers $\op{St}^0_k$ for $k\ge i+1$ are trivial as well.
 \end{proof}

 \section{Hilbert polynomial and Poincar\'e function}\label{S4}

Now we can compute the number of independent differential invariants. Since $G$ acts transitively on $M$,
the codimension of the orbit of $G$ in $J^k(\mathcal{C}_M\!)$ is equal to the codimension of the
orbit of $\mathbb{D}_{k+1}$ in $J^k_a(\mathcal{C}_M\!)$ (where $a\in M$ is a fixed point). 
Denoting the orbit of $G$ through a generic point $a_k\in J^k_a(\mathcal{C}_M\!)$ 
by $\mathcal{O}_k\subset J^k_a(\mathcal{C}_M\!)$ we have:
 $$
\dim(\mathcal{O}_k)=\dim\mathbb{D}_{k+1}-\dim\op{St}_k.
 $$
Notice that
 $$
\op{codim}(\mathcal{O}_k)=\dim J^k_a(\mathcal{C}_M\!)-\dim(\mathcal{O}_k)=\op{trdeg}\mathfrak{F}_k
 $$
is the number of functionally (or, in our context, algebraically) independent scalar differential 
invariants of order $k$.

The Hilbert function is the number of ``pure order" $k$ differential invariants
$H_n(k)=\op{trdeg}\mathfrak{F}_k-\op{trdeg}\mathfrak{F}_{k-1}$.
It is known to be a polynomial (we refer to \cite{KL2} for the proof in our context),
so we will refer to it as the Hilbert polynomial.

The results of Section \ref{S3} and the formulae above allow to compute the values $H_n(k)$,
giving the table in the introduction.

 \begin{theorem}
For $n>3$ the number of ord=$2$ differential invariants is
 $$
H_n(2)=\tfrac1{12}(n^4-13n^2-12),
 $$
the number of ``pure" order $3$ differential invariants is
 $$
H_n(3)=\tfrac1{24}n(n^4+2n^3-5n^2-14n-32),
 $$
and the number of ``pure" order $k>3$ differential invariants is
 $$
H_n(k)=\frac{n(k-1)}2\binom{n+k-1}{k+1}-\binom{n+k-1}k.
 $$
For the exceptional case $n=3$ we have:
 $$
H_3(3)=1,\quad H_3(4)=9,\quad\text{ and }\quad H_3(k)=k^2-4\text{ for }k\ge5.
 $$
 \end{theorem}

Notice that $H_n(k)\sim\tfrac{n^2-n-2}{2}\tfrac1{(n-1)!}\,k^{n-1}$, which confirms the (obvious) fact that
moduli of conformal structures are parametrized by $\binom{n}2-1$ functions of $n$ arguments.

 \begin{proof}
From Lemma \ref{St} we have: $H_n(3)=\dim J^3_a(\mathcal{C}_M\!)-\dim\mathbb{D}_4-H_n(2)$ for $n\ge4$
and $H_3(4)=\dim J^4_a(\mathcal{C}_M\!)-\dim\mathbb{D}_5-H_3(3)$ for $n=3$. These numbers are positive, and
we have by Lemma \ref{St}: $H_n(k)=\dim\g_k-\dim\Delta_{k+1}$ for $k>3$, $n\ge4$ and for $k>4$, $n=3$.
 \end{proof}

The Poincar\'e function is the generating function for the Hilbert polynomial, defined by
$P_n(z)=\sum_{k=0}^\infty H_n(k)z^k$. This is a rational function with the only pole $z=1$ of
order equal to the minimal number of invariant derivations in the Lie-Tresse generating set \cite{KL2}.

Depending on dimension $n>2$ the Poincar\'e functions are:
 \begin{gather*}
P_3(z)=\frac{z^3(1+6z-3z^2-5z^3+3z^4)}{(1-z)^3},\\
P_4(z)=\frac{z^2(3+24z-35z^2+8z^3+9z^4-4z^5)}{(1-z)^4},\\
P_5(z)=\frac{z^2(24+15z-85z^2+74z^3-10z^4-14z^5+5z^6)}{(1-z)^5}.
 \end{gather*}
 $$
P_n(z)=\frac{(n+1)nz-2(n+z)}{2z(1-z)^n}+n\Bigl(\frac1z+z-z^3\Bigr)+
\Bigl(\tbinom{n}{2}+1\Bigr)(1-z^2).
 $$

 \section{Conclusion}\label{S5}

We have described the lowest degree differential invariants: 1 invariant $I_1$ of order $k=3$ for $n=3$,
3 invariants $I_1,I_2,I_3$ of order $k=2$ for $n=4$, etc.
How to see the next invariants of order $(k+1)$?

For $n=3$ there are 9 differential invariants of ``pure" order $4$. They can be extracted from
$\nabla_j(I_1)$ ($\nabla_j=\hat{e_j}$) and the structure constants $c^k_{ij}$ given by $[e_i,e_j]=c_{ij}^ke_k$.
This gives $3+9>9$ invariants of order $4$.

For $n=4$ we can take $\nabla_j(I_i)$, structure constants $c^k_{ij}$,
but we can also add the Christoffel symbols $\Gamma_{ij}^k$ of the
Levi-Civita connection of the normalized metric $g_0$ in the basis $\{e_i\}$ as invariants of order $3$.
This gives $12+24+40>36$ invariants of order $3$.

However the amount of invariants obtained in this way exceeds the number of the independent invariants
we have found. This is because there exist algebraic relations between them (like differential Bianchi's identity).
In terms of Lie-Tresse representation of invariants, these relations are called differential syzygies.
The important problem of understanding these syzygies remains open.



\begin{thebibliography}{11}
\footnotesize

\bibitem{BEG}
T.N. Bailey, M.G. Eastwood, C.R. Graham, {\it Invariant theory for conformal and CR geometry},
Annals of Mathematics {\bf 139}, 491-552 (1994).

\bibitem{B}
A. Besse, {\it Einstein manifolds}, Springer-Verlag, Berlin Heidelberg (1987).

\bibitem{FG}
C. Fefferman, C.R. Graham, {\it Conformal invariants},
\'Elie Cartan et les Math\'ematiques d'Adjourd'hui, Ast\'erisque, 95-116 (1985).

\bibitem{G}
R. Gover, {\it Invariant Theory and Calculus for Conformal Geometries},
Advances in Mathematics {\bf 163}, no.2, 206-257 (2001).

\bibitem{Ha}
C.N. Haskins, {\it On the invariants of quadratic differential forms},
Transactions Amer. Math. Soc. {\bf 3}, 71-91 (1902).

\bibitem{Hi}
D. Hilbert, {Theory of algebraic invariants} (translated from the German original),
Cambridge University Press, Cambridge (1993).

\bibitem{K}
B. Kruglikov, {\it Differential Invariants and Symmetry: Riemannian Metrics and Beyond},
Lobachevskii Journal of Mathematics {\bf 36}, no.3, 292-297 (2015).

\bibitem{KL1}
B. Kruglikov, V. Lychagin, {\it Geometry of Differential equations\/},
Handbook of Global Analysis, Ed. D.Krupka, D.Saunders, Elsevier, 725-772 (2008).

\bibitem{KL2}
B. Kruglikov, V. Lychagin, {\it Global Lie-Tresse theorem}, Selecta Mathematica New Ser.
DOI 10.1007/s00029-015-0220-z (2016).

\bibitem{LY}
V. Lychagin, V. Yumaguzhin, {\it Invariants in Relativity Theory},
Lobachevskii Journal of Mathematics {\bf 36}, no.3, 298-312 (2015).

\bibitem{M}
D. Mumford, J. Fogarty, F. Kirwan, {\it Geometric invariant theory},
Ergebnisse der Mathematik und ihrer Grenzgebiete (2), {\bf 34}, Springer-Verlag, Berlin (1994).

\bibitem{O}
P. Olver, {\it Equivalence, Invariants, and Symmetry}, Cambridge University Press, Cambridge (1995).


\bibitem{R}
M. Rosenlicht, {\it Some basic theorems on algebraic groups}, American Journal of Mathematics
{\bf 78}, 401-443 (1956).

\bibitem{T}
T.Y. Thomas, {\it The Differential Invariants of Generalized Spaces},
Cambridge University Press, Cambridge (1934).

\bibitem{W}
H. Weyl, {\it The Classical Groups}, Princeton University Press, Princeton (1939).

\bibitem{Z}
K. Zorawski, {\it \"Uber Biegungsinvarianten},  Acta Math. {\bf 16}, no.1, 1-64 (1892).

\end{thebibliography}
\end{document}